\documentclass[preprint,12pt]{elsarticle}

\usepackage[normalem]{ulem}
\newcommand{\pp}[2]{\def\firstarg{#1}\def\secondarg{#2}\def\empty{}%
\sout{#1}%
\ifx\firstarg\empty\else\ifx\secondarg\empty\else{ }\fi\fi%
\uwave{#2}}

\usepackage{latexsym}
\usepackage{amsmath}
\usepackage{amssymb}
\usepackage{graphicx}
\usepackage{xcolor}
\usepackage{enumerate,enumitem}
\usepackage{todonotes}

\usepackage[linesnumbered,boxed,ruled]{algorithm2e} 
\SetAlgoInsideSkip{medskip}

\usepackage{tikz}
\usetikzlibrary{calc}
\usepackage{float}

\newtheorem{observation}{Observation}
\newtheorem{theorem}{Theorem}     
\newtheorem{corollary}{Corollary}
\newtheorem{lemma}{Lemma}

\newtheorem{proposition}{Proposition}

\newtheorem{remark}{Remark}
\newtheorem{definition}{Definition}

\newenvironment{proof}{\noindent {\bf Proof.\,}}{\qed}

\def\CL{C\!L}

\def\CLG{\textsc{color-line graph}}
\def\PCLG{\textsc{proper color-line graph}}

\begin{document}

\begin{frontmatter}

\title{Color-line and Proper Color-line Graphs
}

\author{Van Bang Le}
\ead{van-bang.le@uni-rostock.de}
\address{Universit\"at Rostock, Institut f\"ur Informatik, Rostock, Germany}

\author{Florian Pfender}
\ead{florian.pfender@ucdenver.edu}
\address{University of Colorado Denver, Department of Mathematics \& Statistics,\\ Denver, CO 80202, USA}

\begin{abstract}
Motivated by investigations of rainbow matchings in edge colored graphs, we introduce the notion of color-line graphs that generalizes the classical concept of line graphs in a natural way. Let $H$ be a (properly) edge colored graph. The ({\it proper}) \emph{color-line graph} $\CL(H)$ of $H$ has edges of $H$ as vertices, and two edges of $H$ are adjacent in $\CL(H)$ if they have an endvertex in common or have the same color. 

We give Krausz-type characterizations for (proper) color-line graphs, and show that, for any fixed $k$, recognizing color-line graphs of properly edge colored graphs $H$ with at most $k$ colors is polynomially solvable. Moreover, we give a good characterization for proper $2$-color-line graphs that yields a linear time recognition algorithm in this case.

In contrast, we point out that, for any fixed $k\ge 2$, recognizing if a graph is the color-line graph of some graph $H$ in which the edges are colored with at most $k$ colors is NP-complete.
\end{abstract}


\end{frontmatter}

\section{Introduction}

All graphs considered in this paper are simple and have no loops.
The \emph{line graph} $L(H)$ of a graph $H$ has vertex set $E(H)$. Two vertices in $L(H)$ are adjacent if the corresponding edges in $H$ are adjacent in $H$. A graph $G$ is called a line graph if there exists a graph $H$ 
such that $G$ is isomorphic to $L(H)$, we call the graph $H$ a root graph of $G$. The notion of line graphs was first introduced by Whitney~\cite{Whitney} in 1932. 
Line graphs form a very basic graph class in graph theory (especially, one of the five basic classes of perfect graphs is the class of line graphs of bipartite graphs; cf.~\cite{CRST}) and are widely investigated. See~\cite{CaiCP,HeBe,Prisner} for more information on line graphs, their generalizations and related concepts.

Among others, the following three characterizations of line graphs are well-known. (An \emph{odd triangle} $T$ in 
a graph is a $K_3$ such that some vertex has an odd number of neighbors in $T$). 

\begin{theorem}
\label{thm:line1}
The following statements are equivalent for every graph $G$:
\begin{description}
 \item[(A)] $G$ is a line graph.
 \item[(B) (Krausz~\cite{Krau})] There exists a family of cliques $\mathcal{Q}$ of $G$ such that every 
vertex of $G$ belongs to at most two members in $\mathcal{Q}$ and every edge of $G$ belongs to exactly one member 
in $\mathcal{Q}$.
 \item[(C) (van Rooij \& Wilf~\cite{RooWil})] $G$ is claw-free and every two odd triangles having a common edge belong to a $K_4$. 
 \item[(D) (Beineke~\cite{Beineke})] $G$ does not contain any of the nine graphs depicted in Figure~\ref{fig:line-forbidden} as an induced subgraph.
 \end{description}
\end{theorem} 

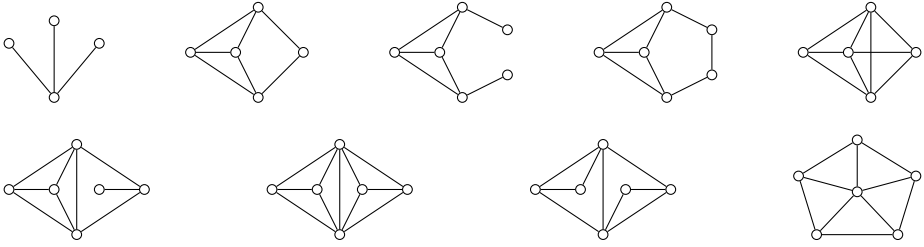
\begin{figure}[thb]
\begin{center}
\begin{tikzpicture}[scale=.6]
\tikzstyle{vertex}=[draw,circle,inner sep=1.3pt] 
\node[vertex] (1) at (0,1.2)   {};
\node[vertex] (2) at (1,1.7)   {};
\node[vertex] (3) at (2,1.2) {};
\node[vertex] (4) at (1,0)   {};

\draw (1) -- (4) -- (2);
\draw (3) -- (4);
\end{tikzpicture}
\hfill
\begin{tikzpicture}[scale=.6]
\tikzstyle{vertex}=[draw,circle,inner sep=1.3pt] 
\node[vertex] (1) at (0,1)   {};
\node[vertex] (2) at (1,1)   {};
\node[vertex] (3) at (1.5,0) {};
\node[vertex] (5) at (2.5,1)   {};
\node[vertex] (6) at (1.5,2) {};

\draw (1) -- (3) -- (5);
\draw (2) -- (3); 

\draw (1) -- (6) -- (5);
\draw (2) -- (6); 
\draw (1) -- (2);
\end{tikzpicture}
\hfill
\begin{tikzpicture}[scale=.6]
\tikzstyle{vertex}=[draw,circle,inner sep=1.3pt] 
\node[vertex] (1) at (0,1)   {};
\node[vertex] (2) at (1,1)   {};
\node[vertex] (3) at (1.5,0) {};
\node[vertex] (a) at (2.5,0.5) {};
\node[vertex] (b) at (2.5,1.5) {};
\node[vertex] (6) at (1.5,2) {};

\draw (1) -- (3) -- (a);
\draw (2) -- (3); 

\draw (1) -- (6) -- (b);
\draw (2) -- (6); 
\draw (1) -- (2);
\end{tikzpicture}
\hfill
\begin{tikzpicture}[scale=.6]
\tikzstyle{vertex}=[draw,circle,inner sep=1.3pt] 
\node[vertex] (1) at (0,1)   {};
\node[vertex] (2) at (1,1)   {};
\node[vertex] (3) at (1.5,0) {};
\node[vertex] (a) at (2.5,0.5) {};
\node[vertex] (b) at (2.5,1.5) {};
\node[vertex] (6) at (1.5,2) {};

\draw (1) -- (3) -- (a);
\draw (2) -- (3); 

\draw (1) -- (6) -- (b);
\draw (2) -- (6); 
\draw (1) -- (2);
\draw (a) -- (b);
\end{tikzpicture}
\hfill
\begin{tikzpicture}[scale=.6]
\tikzstyle{vertex}=[draw,circle,inner sep=1.3pt] 
\node[vertex] (1) at (0,1)   {};
\node[vertex] (2) at (1,1)   {};
\node[vertex] (3) at (1.5,0) {};
\node[vertex] (5) at (2.5,1)   {};
\node[vertex] (6) at (1.5,2) {};

\draw (1) -- (3) -- (5) -- (2) -- (3);

\draw (1) -- (6) -- (5);
\draw (1) -- (2) -- (6);
\draw (3) -- (6); 
\end{tikzpicture}
\end{center}

\begin{center}
\begin{tikzpicture}[scale=.6]
\tikzstyle{vertex}=[draw,circle,inner sep=1.3pt] 
\node[vertex] (1) at (0,1)   {};
\node[vertex] (2) at (1,1)   {};
\node[vertex] (3) at (1.5,0) {};
\node[vertex] (4) at (2,1)   {};
\node[vertex] (5) at (3,1)   {};
\node[vertex] (6) at (1.5,2) {};

\draw (1) -- (3) -- (5);
\draw (2) -- (3);

\draw (1) -- (6) -- (5);
\draw (2) -- (6);
\draw (1) -- (2);
\draw (4) -- (5);
\draw (3) -- (6); 
\end{tikzpicture}
\hfill
\begin{tikzpicture}[scale=.6]
\tikzstyle{vertex}=[draw,circle,inner sep=1.3pt] 
\node[vertex] (1) at (0,1)   {};
\node[vertex] (2) at (1,1)   {};
\node[vertex] (3) at (1.5,0) {};
\node[vertex] (4) at (2,1)   {};
\node[vertex] (5) at (3,1)   {};
\node[vertex] (6) at (1.5,2) {};

\draw (1) -- (3) -- (5);
\draw (2) -- (3) -- (4);

\draw (1) -- (6) -- (5);
\draw (2) -- (6) -- (4);
\draw (1) -- (2);
\draw (4) -- (5);
\draw (3) -- (6); 
\end{tikzpicture}
\hfill
\begin{tikzpicture}[scale=.6]
\tikzstyle{vertex}=[draw,circle,inner sep=1.3pt] 
\node[vertex] (1) at (0,1)   {};
\node[vertex] (2) at (1,1)   {};
\node[vertex] (3) at (1.5,0) {};
\node[vertex] (4) at (2,1)   {};
\node[vertex] (5) at (3,1)   {};
\node[vertex] (6) at (1.5,2) {};

\draw (1) -- (3) -- (5);
\draw (3) -- (4);

\draw (1) -- (6) -- (5);
\draw (2) -- (6);
\draw (1) -- (2);
\draw (4) -- (5);
\draw (3) -- (6);
\end{tikzpicture}
\hfill
\begin{tikzpicture}[scale=.6]
\tikzstyle{vertex}=[draw,circle,inner sep=1.3pt] 
\node[vertex] (1) at (0.1,-0.1)  {};
\node[vertex] (2) at (-0.3,1.2)  {};
\node[vertex] (3) at (1,2)       {};
\node[vertex] (4) at (2.3,1.2)   {};
\node[vertex] (5) at (1.9,-0.1)  {};
\node[vertex] (6) at (1,.85)       {};

\draw (1) -- (2) -- (3) -- (4) -- (5) -- (1) -- (6) -- (2);
\draw (3) -- (6) -- (4);
\draw (5) -- (6);
\end{tikzpicture}
\end{center}
\caption{Forbidden induced subgraphs for line graphs.}
\label{fig:line-forbidden}
\end{figure}

Each of the characterizations (C) and (D) in Theorem~\ref{thm:line1} implies that line graphs can be recognized in polynomial time. 
In fact, the following stronger result is also well-known.

\begin{theorem}[Roussopoulos~\cite{Rou}, Lehot~\cite{Lehot}]\label{thm:line2}
Line graphs can be recognized in linear time. Moreover, if $G$ is a line graph, a root graph $H$ with $G=L(H)$ can be found in linear time, too.
\end{theorem}

Line graphs are also interesting and useful as they relate questions about edges with questions about vertices. By definition, for instance, a matching in $H$ corresponds to an independent vertex set (of the same size) in $L(H)$, and vice versa. 

Given a graph $H=(V(H),E(H))$, an edge coloring is a function 
$\phi : E(H)\rightarrow \mathcal{C}$ mapping each edge 
$e\in E(H)$ to a color $\phi(e)\in \mathcal{C}$; $\phi$ is a \emph{proper} edge coloring 
if, for all distinct edges $e$ and $e'$, $\phi(e)\not=\phi(e')$ whenever $e$ and $e'$ have 
an endvertex in common. A (properly) edge colored graph $(H,\phi)$ is a pair of a graph 
together with a (proper) edge coloring. A {\em rainbow subgraph\/} of an edge colored graph 
is a subgraph whose edges have distinct colors. Rainbow subgraphs appear frequently in the 
literature, for a recent survey we point to~\cite{KL}. 

The following notion is motivated by investigations on rainbow matchings (see~\cite{LP}). 

\begin{definition} Let $H$ be a (not necessarily properly) edge colored graph. The {\em color-line graph\/} $\CL(H)$ of $H$ has the edge set $E(H)$ as its vertex set and two vertices in $\CL(H)$ are adjacent if the corresponding edges in $H$ are adjacent or have the same color. 
A graph $G$ is a {\em color-line graph\/} if there exists an edge colored graph $H$ such that $G$ is isomorphic to $\CL(H)$. 
Further, a graph $G$ is a {\em proper color-line graph\/} if there exists a properly edge colored graph $H$ such that $G$ is isomorphic to $\CL(H)$. 
\end{definition}

See Figure~\ref{fig:color-line} for an example. 
By definition, rainbow matchings in $H$ correspond to independent vertex sets in $\CL(H)$, and vice versa. 

In this paper, we study color-line and proper color-line graphs in their own right, primarily from the point of view of recognizing if a given graph is a color-line graph, respectively, a proper color-line graph. 
In Section~\ref{sec:color-line} we discuss properties of color-line graphs and proper color-line graphs. In particular we give Krausz-type characterizations for color-line and proper color-line graphs. In Sections~\ref{sec:k-color-line} and~\ref{sec:proper-k-color-line} we discuss the recognition problem for $k$-color-line graphs and proper $k$-color-line graphs. 
These are color-line graphs of edge colored graphs with at most $k$ colors. 
We show that, for any fixed $k\ge 2$, recognizing if a graph is a $k$-color-line graph is NP-complete, while recognizing if a graph is a proper $k$-color-line graph is polynomially solvable. In Section~\ref{sec:proper-2-color-line} we show that proper $2$-color-line graphs admit a characterization by forbidden induced subgraphs and can be recognized in linear time. Section~\ref{sec:open} lists some open problems for further research.

\paragraph{Definitions and notation} In a graph, a set of vertices is a {\em clique\/}, respectively, an {\em independent set\/} if every pair, respectively, no pair of vertices in this set is adjacent. A \emph{co-bipartite} graph is the complement of a bipartite graph. 
The complete graph with $n$ vertices is denoted by $K_n$; $K_3$ is also called a {\it triangle}, and $K_n-e$ is the graph obtained from $K_n$ by deleting an edge. The path and cycle with $n$ vertices 
are denoted by $P_n$ and $C_n$, respectively. We use $K_{p,q}$ to denote the complete bipartite graph with $p$ vertices of one color class and $q$ vertices of the second color class. A {\it biclique\/} is a $K_{p,q}$ for some $p,q$, and a {\it $q$-star\/} is a $K_{1,q}$; the $3$-star $K_{1,3}$ is also called a {\it claw\/}. 

Let $G = (V, E)$ be a graph. For a vertex $v\in V$ we write $N(v)$ for the set of its neighbors in $G$. A \emph{universal\/} vertex $v$ is one such that $N(v)\cup\{v\}=V$. An {\it $r$-regular graph\/} is one in which each vertex has degree exactly $r$.

For a subset $U\subseteq V$ we write $G[U]$ for the subgraph of $G$ induced by $U$ and $G-U$ for the graph $G[V -U]$; for a vertex $v$ we write $G-v$ rather than $G[V\setminus \{v\}]$.  
For a subset $F\subseteq E(G)$, $G\setminus F$ is the spanning subgraph of $G$ obtained from $G$ by deleting all edges in $F$; for an edge $e$ we write $G-e$ rather than $G\setminus\{e\}$. Sometimes we identify a set of vertices $A\subseteq V(G)$ with the subgraph induced by $A$. Thus, $E(A)$ means $E(G[A])$ and, for an edge $e$, $e\in A$ means $e\in E(A)$. A set of vertices is {\it nontrivial} if it contains at least two vertices.

\section{Color-line Graphs and Proper Color-line Graphs}\label{sec:color-line}

In this section we discuss some basic properties of color-line and proper color-line graphs. Among other things we give Krausz-type characterizations for these graphs.

The following basic facts follow immediately from definition; we often use these facts without reference:

\begin{enumerate}
  \item For any edge colored graph $H$, $L(H)$ is a spanning subgraph of $\CL(H)$.\label{fact1} 
  \item Every line graph is a proper color-line graph (consider an edge coloring of a root graph $H$ in which each edge has its own color).\label{fact2} 
  \item Induced subgraphs of a (proper) color-line graph are again (proper) color-line graphs.\label{fact3} 
  \item If $H'$ is a subgraph of $H$ then $\CL(H')$ is an induced subgraph of $\CL(H)$.\label{fact4} 
  \item $G$ is a (proper) color-line graph if and only if each component of $G$ is a (proper) color-line graph.\label{fact5}  
\end{enumerate}


\begin{theorem}[Krausz-type characterization for color-line graphs]\label{thm:color-line}
A graph $G$ is a color-line graph if and only if there exists a family of cliques $\mathcal{Q}$ in $G$ such that
  \begin{itemize}
    \item[\em (a)] every vertex of $G$ belongs to at most three members in $\mathcal{Q}$, and every edge of $G$ belongs to at least one and at most two members in $\mathcal{Q}$,
    \item[\em (b)] there are two mappings $F \rightarrow \mathcal{Q}$, $e \mapsto Q_e$, and $W \rightarrow \mathcal{Q}$, $v \mapsto Q_v$, where $F\subseteq E(G)$ and $W\subseteq V(G)$ is the set of all edges belonging to exactly two members, respectively, the set of all vertices belonging to exactly three members in $\mathcal{Q}$, such that
     \begin{itemize}
       \item[\em (b1)] for all $x\in F\cup W$, $x\in Q_x$, and
       \item[\em (b2)] for all $x, y\in F\cup W$, if $x\in Q_y$ then $Q_x=Q_y$ and if $x\not\in Q_y$ then $Q_x\cap Q_y=\emptyset$.
     \end{itemize}
  \end{itemize}
\end{theorem}
\begin{proof} 
\emph{Necessity}: 
Let $G=\CL(H)$ with respect to an (arbitrary) edge coloring $\phi$ of $H$. Then, for each vertex $v=ab\in E(H)=V(G)$, each of the three sets 
\begin{align*}
Q_v &= \{e\in E(H)\mid \phi(e)=\phi(ab)\},\\
Q_a &= \{e\in E(H)\mid a\in e\}, \text{ and}\\
Q_b &=\{e\in E(H)\mid b\in e\},
\end{align*}
is a clique in $G$. Let $\mathcal{Q}$ consist of all such nontrivial cliques.  Moreover, 
\begin{equation*}
W = \{v\in V(G) \mid v=ab\in E(H), \text{ $Q_v, Q_a$, and $Q_b$ are nontrivial}\}
\end{equation*} 
is the set of all vertices of $G$ belonging to exactly three members in $\mathcal{Q}$, and 
\begin{equation*}
F = \{uv\in E(G)\mid \text{$u=ab, v=bc\in E(H), \phi(ab)=\phi(bc)$}\}
\end{equation*} 
is the set of all edges of $G$ belonging to exactly two members in $\mathcal{Q}$,  namely $Q_{u}=Q_v$ and $Q_b$; note that $u$ and $v$ need not to belong to $W$. 
With these definitions (a) is satisfied. Further, the mappings 
\begin{equation*}
v\mapsto Q_v,\, v\in W,
\end{equation*}
and 
\begin{equation*}
e\mapsto Q_e:= Q_u = Q_v,\, e=uv\in F,
\end{equation*} 
 satisfy (b). 

\smallskip
\emph{Sufficiency}: Let $\mathcal{Q}$ be a family of cliques in $G$ which satisfy conditions (a) and (b). Write
\begin{align*}
G'           &= G\setminus\{e\in Q_x\mid e\text{ is only in $Q_x$ for some  $x\in F\cup W$}\},\\
\mathcal{Q}' &= \mathcal{Q}\setminus\{Q_x\mid x\in F\cup W\}. 
\end{align*} 
Then each member of $\mathcal{Q}'$ is a clique in $G'$. Moreover, each edge of $G'$ is in 
exactly one member of $\mathcal{Q}'$ and each vertex of $G'$ is in at most two 
members in $\mathcal{Q}'$ (by (a) and (b1)). 

By Theorem~\ref{thm:line1} (B), $G'=L(H)$ for some graph $H$. Consider the following edge 
coloring $\phi$ of $H$; note that $V(G')=V(G)=E(H)$.
\begin{itemize}
 \item For each $x\in F\cup W$, all edges of $H$ corresponding to vertices of $Q_x$ have the same color; different colors by different $Q_x$'s.
 \item Each other edge of $H$ has its own color.   
\end{itemize}
By (b2), $\phi$ is well-defined. Moreover, with respect to $\phi$, $G=\CL(H)$. 
Indeed, 
\begin{eqnarray*}
 e_1e_2\in E(G)\, & \Leftrightarrow & \, e_1e_2\in E(G')=E(L(H)) \text{ or } e_1e_2\in Q_x \text{ for some $x\in F\cup W$}\\
                &\Leftrightarrow   & \, \text{$e_1$ and $e_2$ are incident in $H$ or $\phi(e_1)=\phi(e_2)$}\\
                &\Leftrightarrow   & \, e_1e_2\in E(\CL(H))
\end{eqnarray*} 
\end{proof}

\smallskip
As an example, the graph $G$ depicted in Figure~\ref{fig:color-line} on the left side has a family $\mathcal{Q}$ satisfying the conditions in Theorem~\ref{thm:color-line}, consisting of the cliques $\{1,2,3,4,5,6\}$, $\{2,3,4,5,6,7\}$, $\{1,8\}$, $\{1,9\}$, $\{7,8\}$ and $\{7,9\}$, with $F=\{23,24,25,26,34,35,36,45,46,56\}$, $W=\{1,7\}$ and $Q_1=\{1,8\}$, $Q_7=\{2,3,4,5,6,7\}=Q_e$ for all $e\in F$. 

The graph $H$ constructed in the proof of Theorem~\ref{thm:color-line} is depicted on the right; the edges of $H$ are colored with three colors as indicated. 
   
\begin{figure}[thb]
\begin{center}
\begin{tikzpicture}[scale=.5] 
\tikzstyle{vertex}=[draw,circle,inner sep=1pt] 
\node[vertex] (1) at (5,1) {\tiny $1$};
\node[vertex] (2) at (1.5,3) {\tiny $2$};
\node[vertex] (3) at (3,3) {\tiny $3$};
\node[vertex] (4) at (4,5) {\tiny $4$};
\node[vertex] (5) at (2,6) {\tiny $5$};
\node[vertex] (6) at (0,5) {\tiny $6$};
\node[vertex] (7) at (5,8) {\tiny $7$};
\node[vertex] (8) at (6.5,4) {\tiny $8$};
\node[vertex] (9) at (8,4) {\tiny $9$};

\draw (1) -- (2); \draw (1) -- (3); \draw (1) -- (4); \draw (1) -- (5);
\draw (1)to[bend angle=40, bend left](6); \draw (1) -- (8); \draw (1) -- (9);
 
\draw (7) -- (2); \draw (7) to[bend angle=30, bend left](3); \draw (7) -- (4); \draw (7) -- (5);
\draw (7) to[bend angle=20, bend right](6); \draw (7) -- (8); \draw (7) -- (9);

\draw (2) -- (3) -- (4) -- (5) -- (6) -- (2) -- (4) -- (6) -- (3) -- (5) -- (2);
\end{tikzpicture}
\qquad\qquad
\begin{tikzpicture}[scale=.5] 
\tikzstyle{vertex}=[draw,circle,inner sep=1.5] 
\node[vertex] (1) at (0,3) {};  
\node[vertex] (2) at (2,3) {}; 
\node[vertex] (3) at (4,3) {}; 
\node[vertex] (4) at (6,3) {}; 
\node[vertex] (5) at (8,3) {}; 
\node[vertex] (6) at (5,0) {}; 
\node[vertex] (7) at (9,0) {}; 
\node[vertex] (8) at (10,3) {}; 
\node[vertex] (9) at (6,6) {}; 
\node[vertex] (10) at (10,6) {}; 

\draw[line width=2pt, draw=lightgray] (6) -- node[left] {\tiny $6$\,\,\,} (1); 
\draw[line width=2pt, draw=lightgray] (6) -- node[right] {\tiny $5$} (2); 
\draw[line width=2pt, draw=lightgray] (6) -- node[right] {\tiny \!$4$} (3); 
\draw[line width=2pt, draw=lightgray] (6) -- node[right] {\tiny \!$3$} (4); 
\draw[line width=2pt, draw=lightgray] (6) -- node[right] {\tiny $2$} (5); 
\draw[line width=2pt, draw=lightgray] (8) -- node[left] {\tiny $7$\,\,\,}(9); 

\draw[line width=1.5pt, style=densely dashed, draw=brown] (6) -- node[below] {\tiny $1$} (7); 
\draw[line width=1.5pt, style=densely dashed, draw=brown] (9) -- node[above] {\tiny $8$} (10);

\draw[line width=2.5pt] (7) -- node[right] {\tiny $9$} (8);  
\end{tikzpicture}
\end{center}
\caption{An example illustrating Theorem~\ref{thm:color-line}: The graph $G$ (left) and a color-line root graph $H$ (right).}
\label{fig:color-line}
\end{figure}
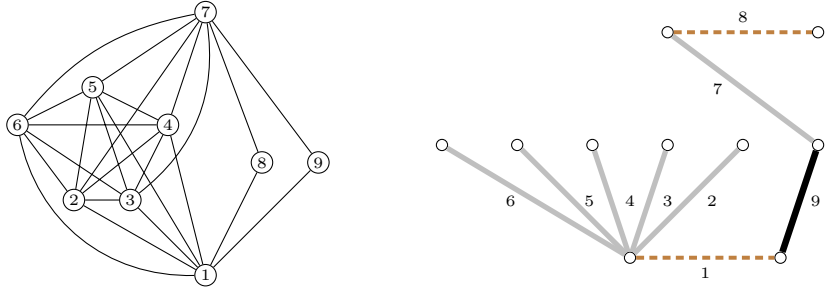

\begin{remark}\label{rem:1}
In case $F=W=\emptyset$, Theorem~\ref{thm:color-line} is precisely  Krausz's characterization for line graphs (Theorem~\ref{thm:line1} (B)). 
\end{remark}

\begin{remark}\label{rem:2}
Theorem~\ref{thm:color-line} implies that $K_{1,4}$ is the smallest graph that is not a color-line graph. Thus, color-line graphs are $K_{1,4}$-free.

Moreover, not every co-bipartite graph, hence not every claw-free graph, is a color-line graph. The complement of $6K_2$ is such an example.
\end{remark}

\noindent
Recall that a \emph{proper color-line graph} is a color-line graph $G=\CL(H)$ of some \emph{properly} edge colored graph $H$. 
Clearly, proper color-line graphs are color-line graphs (but not the converse). 

In the proof of the necessity part of Theorem~\ref{thm:color-line}, the set $F$ is empty in case the edge coloring $\phi$ of $H$ is proper. Hence Theorem~\ref{thm:color-line} can be stated in a slightly simpler form for proper color-line graphs as follows.

\begin{theorem}[Krausz-type characterization for proper color-line graphs]\label{thm:proper-color-line1}
A graph $G$ is a proper color-line graph if and only if there exists a family of cliques $\mathcal{Q}$ in $G$ such that
  \begin{itemize}
    \item[\em (a)] every vertex of $G$ belongs to at most three members, and every edge of $G$ belongs to exactly one member in $\mathcal{Q}$,
    \item[\em (b)] there is a mapping $W \rightarrow \mathcal{Q}$, $v \mapsto Q_v$, where $W\subseteq V(G)$ is the 
        set of all vertices belonging to exactly three members in $\mathcal{Q}$, such that
     \begin{itemize}
       \item[\em (b1)] $v\in Q_v$, and 
       \item[\em (b2)] for all $u, v\in W$, if $u\in Q_v$ then $Q_u=Q_v$ and if $u\not\in Q_v$ then 
                       $Q_u\cap Q_v=\emptyset$. 
     \end{itemize}
  \end{itemize}
\end{theorem}

\begin{proof} It remains to discuss the sufficiency part in the proof of Theorem~\ref{thm:color-line}. Let $\mathcal{Q}$ be a family of cliques in $G$ which satisfy the conditions (a) and (b), and define
\begin{align*}
G'           &= G\setminus\{e\in Q_x\mid x\in W\},\\
\mathcal{Q}' &= \mathcal{Q}\setminus\{Q_x\mid x\in W\}. 
\end{align*} 
Then, as in the proof of Theorem~\ref{thm:color-line}, $G'=L(H)$ for some graph $H$. Moreover, the edge coloring $\phi$ of $H$ defined there is proper because now, for each $v\in W$, the edges in $H$ correspondng to vertices in $Q_v$ form a matching in $H$. 
\end{proof}

\smallskip
Again, in case $W=\emptyset$, Theorem~\ref{thm:proper-color-line1} is precisely  Krausz's characterization for line graphs (Theorem~\ref{thm:line1} (B)).
 
The family $\mathcal Q$ in Theorem~\ref{thm:proper-color-line1} is a certain edge clique partition of $G$. Proper color-line graphs also admit a characterization in terms of a certain vertex clique  partition as follows. 
Let $\phi$ be an edge coloring of a graph $H$. An edge $xy$ of $\CL(H)$ is an L-edge if the corresponding edges $x$ and $y$ in $H$ are incident and is a C-edge if $\phi(x)=\phi(y)$. In general, an edge in $\CL(H)$ may be an L-edge and a C-edge at the same time (these edges form the set $F$ in Theorem~\ref{thm:color-line}). The difference between $\CL(H)$ and $L(H)$ in case where the edge coloring $\phi$ of $H$ is proper is that, in $\CL(H)$, the set of all L-edges is disjoint from the set of all C-edges. The following characterization of proper color-line graphs is basically this fact.
   
\begin{theorem}\label{thm:proper-color-line2}
A graph $G$ is a proper color-line graph if and only if $G$ admits a vertex clique partition $\mathcal C$, $V(G)=\bigcup_{Q\in\mathcal C} Q$, such that the graph $G\setminus\big(\bigcup_{Q\in\mathcal C} E(Q)\big)$ is a line graph.
\end{theorem}
\begin{proof}
\emph{Necessity}: 
Let $G=\CL(H)$ with respect to a proper edge coloring $\phi$ of $H$. Then, for each color $c$ used by $\phi$, \begin{equation*}
Q_c= \{e\in E(H)\mid \phi(e)=c\}
\end{equation*}
is a clique in $G$, and $\mathcal C=\{Q_c\mid \text{color $c$ used by } \phi\}$ is a partition of $V(G)$ into cliques $Q_c$. As $\phi$ is a proper coloring, each $Q_c$ is a matching in $H$. Hence, by definition of $\CL(H)$ and $L(H)$, $G\setminus\bigcup_{Q\in\mathcal C} E(Q)$ is the line graph $L(H)$.

\smallskip
\emph{Sufficiency}: Let $G$ have a clique partition $V(G)= Q_1\cup\ldots \cup Q_r$ such that $G'=G\setminus \bigcup_{i} E(Q_i)$ is the line graph $L(H')$ for some graph $H'$. For each $i$, color all edges of $H'$ corresponding to vertices in $Q_i$ with the same color $i$. Note that this edge coloring of $H'$ is well defined (as each vertex of $G$ belongs to exactly one $Q_i$) and proper (as, for each $i$, $Q_i$ is an independent set in $G'$, hence edges corresponding to vertices in $Q_i$ form a matching in $H'$). Letting $H$ denote the obtained edge colored graph, $G=\CL(H)$. 
\end{proof}

\smallskip
As an example, consider the graph $G=K_6-e$ with vertices $v_1, \ldots, v_6$, and $v_1$ and $v_6$ are the two non-adjacent vertices. Then the edge clique  partition $\mathcal{Q}$ consisting of $\{v_1,v_2,v_3\}$, $\{v_1,v_4,v_5\}$, $\{v_3,v_4,v_6\}$, $\{v_2,v_5,v_6\}$, $\{v_2,v_4\}$ and $\{v_3,v_5\}$ satisfies the conditions of Theorem~\ref{thm:proper-color-line1}, and the vertex clique partition $\mathcal{C}$ consisting of $\{v_1\}$, $\{v_6\}$, $\{v_2,v_3\}$ and $\{v_4,v_5\}$ satisfies the conditions of Theorem~\ref{thm:proper-color-line2}. Both $\mathcal{Q}$ and $\mathcal{C}$ give the proper color-line root graph $H=K_4$, properly colored with four colors, for $G$.  

Recall that line graphs are proper color-line graphs (this follows also from Theorem~\ref{thm:proper-color-line2} by considering the vertex clique partition consisting of all one-vertex cliques). Another example is the following:
\begin{corollary}\label{cor:cubic} 
Cubic graphs without bridges are proper color-line graphs.
\end{corollary}
\begin{proof}
By a classical theorem of Petersen~\cite{Pet}, any bridgeless $3$-regular graph $G$ contains a perfect matching $M$. Then, $G\setminus M$ is $2$-regular and thus a line graph.
\end{proof}

\smallskip
Corollary~\ref{cor:cubic} is best possible in the sense that not every cubic graph is a color-line graph. An example is depicted in Figure~\ref{fig:subcubic}.

\begin{figure}[H] 
\begin{center}
\begin{tikzpicture}[scale=.5] 
\tikzstyle{vertex}=[draw,circle,inner sep=1.3pt] 
\node[vertex] (1) at (4,2) {};
\node[vertex] (2) at (2,4) {};
\node[vertex] (3) at (0,2) {};
\node[vertex] (4) at (2,0) {};
\node[vertex] (5) at (3,2) {};
\node[vertex] (6) at (2,3) {};
\node[vertex] (7) at (1,2) {};
\node[vertex] (8) at (2,1) {};
\node[vertex] (9) at (2,2) {};

\draw (1) -- (2) -- (3) -- (4) -- (1);
\draw (5) -- (6) -- (7) -- (8) -- (5);
\draw (2) -- (6); \draw (4) -- (8);
\draw (5) -- (9) -- (7);
\draw (3) to[bend angle=30, bend left] (9);

\node[vertex] (a) at (6,2) {};
\node[vertex] (b) at (8,0) {};
\node[vertex] (c) at (10,2) {};
\node[vertex] (d) at (8,4) {};
\node[vertex] (e) at (7,2) {};
\node[vertex] (f) at (8,1) {};
\node[vertex] (g) at (9,2) {};
\node[vertex] (h) at (8,3) {};
\node[vertex] (i) at (8,2) {};

\draw (a) -- (b) -- (c) -- (d) -- (a);
\draw (e) -- (f) -- (g) -- (h) -- (e);
\draw (b) -- (f); \draw (d) -- (h);
\draw (e) -- (i) -- (g);
\draw (i) to[bend angle=30, bend left] (c);

\node[vertex] (x) at (5,2) {};
\draw (1) -- (x) -- (a);

\node[vertex] (x1) at (5,3) {};
\node[vertex] (x2) at (7,5) {};
\node[vertex] (x3) at (5,7) {};
\node[vertex] (x4) at (3,5) {};
\node[vertex] (x5) at (5,4) {};
\node[vertex] (x6) at (6,5) {};
\node[vertex] (x7) at (5,6) {};
\node[vertex] (x8) at (4,5) {};
\node[vertex] (x9) at (5,5) {};

\draw (x1) -- (x2) -- (x3) -- (x4) -- (x1);
\draw (x5) -- (x6) -- (x7) -- (x8) -- (x5);
\draw (x2) -- (x6); \draw (x4) -- (x8);
\draw (x5) -- (x9) -- (x7);
\draw (x) -- (x1);
\draw (x3) to[bend angle=30, bend left] (x9);

\end{tikzpicture}
\end{center}
\caption{A cubic graph that is not a color-line graph.}
\label{fig:subcubic}
\end{figure}
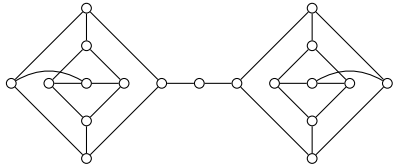

\begin{proposition}\label{prop:K7-e}
$K_7-e$ is not a proper color-line graph. 
\end{proposition}
\begin{proof}
By case analysis, using Theorem~\ref{thm:proper-color-line2}, it can be shown that, for any vertex clique partition $\mathcal C$ of $G=K_7-e$, $G\setminus\big(\bigcup_{Q\in\mathcal C} E(Q)\big)$ is not a line graph (it contains an induced claw or an induced $K_5-e$). 
\end{proof}

\smallskip 
Since a graph may have exponentially many cliques, it is not clear how the characterizations given in Theorems~\ref{thm:color-line}, \ref{thm:proper-color-line1} and~\ref{thm:proper-color-line2} could imply polynomial time recognition algorithms for color-line graphs and proper color-line graphs. Thus, the computational complexity of the following problems is still open. 

\medskip\noindent
\CLG\\[.7ex]
\begin{tabular}{l l}
{\em Instance:}& A graph $G=(V,E)$.\\
{\em Question:}& Is $G$ a color-line graph, i.e., is $G=\CL(H)$ for some\\ 
               & edge colored graph $H$?\\
\end{tabular}

\bigskip\noindent
\PCLG\\[.7ex]
\begin{tabular}{l l}
{\em Instance:}& A graph $G=(V,E)$.\\
{\em Question:}& Is $G$ a proper color-line graph, i.e., is $G=\CL(H)$ for some\\               & properly edge colored graph $H$?\\
\end{tabular}

\medskip
A first natural step in studying the complexity of \CLG\ and of \PCLG\ is to specify the number of colors in the (proper) color-line root graph. This leads to the notion of (proper) $k$-color-line graphs which we will address in the next sections.

\section{Recognizing $k$-Color-line Graphs is Hard}\label{sec:k-color-line}
Let $k\ge 1$ be an integer. A graph $G$ is said to be a ({\em proper\/}) {\em $k$-color-line graph\/} if $G$ is (isomorphic to) the color-line graph of some ({\em properly\/}) {\em edge $k$-colored graph\/} $H$. For example, the graph $G$ depicted in Figure~\ref{fig:color-line} is a $3$-color-line graph and $K_6-e$ is a proper $4$-color-line graph (cf. example after the proof of Theorem~\ref{thm:proper-color-line2}). 

By definition, $1$-color-line graphs are precisely the complete graphs. We are going to point out that, for any fixed $k\ge 2$, the decision problem

\medskip\noindent
\textsc{$k$-color-line graph}\\[.7ex]
\begin{tabular}{l l}
{\em Instance:}& A graph $G=(V,E)$.\\
{\em Question:}& Is $G$ a $k$-color-line graph, i.e., is $G=\CL(H)$ for some\\
               & edge $k$-colored graph $H$?\\
\end{tabular}

\medskip\noindent
is NP-complete. 

First, the following observation allows us to consider the case $k=2$ only; the graph $G+K_1$ is obtained from $G$ by adding a new (isolated) vertex. 

\begin{observation}\label{obs:ktok+1}
$G$ is a (proper) $k$-color-line graph if and only if $G+K_1$ is a (proper) $k+1$-color-line graph. 
\end{observation}
\begin{proof} 
Let $v$ be the new vertex in $G+K_1$. Suppose $G=\CL(H)$ with an edge coloring $\phi$ of $H$. Then $G+K_1=\CL(H')$, where $H'$ is obtained from $H$ by adding a new (isolated) edge $e$ and extending $\phi$ to $\phi'$ by coloring $e$ with a new color. Clearly, if $\phi$ is a (proper) $k$-coloring, then $\phi'$ is a (proper) $(k+1)$-coloring. Conversely, suppose $G+K_1=\CL(H)$ with an edge coloring $\phi$ of $H$, and let $e$ be the edge in $H$ corresponding to the vertex $v$ in $G+K_1$. Then $e$ is an isolated edge in $H$ and $\phi(e)$ occurs only on $e$. Thus $G=\CL(H')$, where $H'=H\setminus\{e\}$ with the restriction $\phi'$ of $\phi$ on $E(H)\setminus\{e\}$. Clearly, if $\phi$ is a (proper) $(k+1)$-coloring, then $\phi'$ is a (proper) $k$-coloring. 
\end{proof}

\begin{theorem}\label{thm:2-color-line-NPc}
For any fixed $k\ge 2$, \textsc{$k$-color-line graph} 
is NP-complete. 
\end{theorem}
\begin{proof}
Note that \textsc{$k$-color-line graph} is obviously in NP. By Observation~\ref{obs:ktok+1} it is enough to see that recognizing $2$-color-line graphs is NP-complete. To see this, we reduce \textsc{line bigraph} to \textsc{$2$-color-line graph}. Let $H_i=(V, E_i)$, $i=1,2$, be two graphs with the same vertex set. The bipartite graph $B$ with the disjoint union $E_1\cup E_2$ as vertex set and two vertices $e_1\in E_1$, $e_2\in E_2$ are adjacent in $B$ if and only if $e_1\cap e_2\not=\emptyset$ is called the {\it line bigraph\/} of $H_1$ and $H_2$. In~\cite{Prisner2003} it is shown that \textsc{line bigraph}, that is to decide if a given bipartite graph is the line bigraph of some graphs $H_1$ and $H_2$, is NP-complete, and \textsc{line bigraph} remains NP-complete in case $E_1\cap E_2=\emptyset$. 

Now, given a bipartite graph $B=(X\cup Y, E)$, let $G$ be obtained from $B$ by completing each $X$ and $Y$ to a clique and by adding two new vertices $x$ and $y$ and adding all edges between $x$ and vertices in $X$ and between $y$ and vertices in $Y$. We claim that $B$ is the line bigraph of two graphs $H_i=(V,E_i)$ with $E_1\cap E_2=\emptyset$ if and only if $G$ is a $2$-color-line graph. 

Indeed, let $B$ be the line bigraph of two graphs $H_i$ with $E_1\cap E_2=\emptyset$. Then let $H$ be obtained from $H_1$ and $H_2$ by adding two new isolated edges $e_x$ and $e_y$ and coloring all edges in $E_1\cup\{e_x\}$ with one color and all edges in $E_2\cup\{e_y\}$ with the second color. Obviously, $G$ is (isomorphic to) the $2$-color-line graph $\CL(H)$ of $H$.

Conversely, suppose $G=\CL(H)$ for some edge $2$-colored graph $H$. 
By construction, $G$ has a unique vertex partition into two cliques. This partition must coincide with the $2$-coloring
of the edges in $H$, say $E(H)=E_1\cup E_2$. We may assume by symmetry that $e_x\in E_1$ and $e_y\in E_2$, where $e_x$ and $e_y$ are the root edges corresponding to $x$ and $y$, respectively. 
Thus, $B$ is the line bigraph of $H_1=(V(H),E_1-e_x)$ and  $H_2=(V(H),E_2-e_y)$.
\end{proof}

\section{Recognizing Proper $k$-Color-line Graphs}\label{sec:proper-k-color-line}
In contrast to the hardness of recognizing $k$-color-line graphs, we show in this section that, for any fixed $k$, recognizing proper $k$-color-line graphs is polynomial. 

Our approach is based on Theorem~\ref{thm:proper-color-line2}, which can be reformulated for proper $k$-color-line graphs as follows.

\begin{theorem}\label{thm:proper-k-color-line}
A graph $G$ is a proper $k$ color-line graph if and only if $G$ admits a vertex clique partition $V(G)=Q_1\cup \ldots \cup Q_k$, such that the graph $G\setminus\big(\bigcup_{1\le i\le k} E(Q_i)\big)$ is a line graph.
\end{theorem}

The key fact in our recognition algorithm is that big enough cliques in a proper $k$-color-line graph must appear in any clique partition as in Theorem~\ref{thm:proper-k-color-line}. More precisely, we show:
\begin{lemma}\label{lem:key}
Suppose $G$ is a proper $k$-color-line graph, and let $Q$ be a maximal clique in $G$ with at least $k+4$ vertices. Then $G-Q$ is a proper $(k-1)$-color-line graph.

Moreover, if $V(G)=Q_1\cup Q_2\cup\ldots \cup Q_k$ is an arbitrary vertex clique partition of a graph $G$ as in Theorem~\ref{thm:proper-k-color-line}, then $Q$ is one of the cliques, $Q_i$, for some $i\in\{1,\ldots, k\}$. 
\end{lemma}
\begin{proof}
Let $G=\CL(H)$ for some graph $H$ with a proper edge $k$-coloring $\phi$.
Write $Q=\{v_1, \ldots, v_s\}$ with $s=|Q|\ge k+4$. To simplify notation, we will also denote the corresponding $Q$-edges in $H$ by $v_i$. Note that for $i\not= j$, $v_i$ and $v_j$ have a common vertex in $H$ (and then $\phi(v_i)\not=\phi(v_j)$), or else $\phi(v_i)=\phi(v_j)$. Hence, in $H$, if $v_i$ and $v_j$ have a common vertex, then any $v_k$ is disjoint from at most one of $v_i$, $v_j$. 

Now assume that two edges $v_i$ and $v_j$ have a common vertex $v$ in $H$. Since $H$ has at most $k$ edges at $v$, at least $4$ of the $Q$-edges do not contain $v$. Let $v_k$ be one such edge. Then either
\begin{enumerate}
\item $v_k\cap v_i\ne \emptyset$ and  $v_k\cap v_j\ne \emptyset$, or
\item $v_k\cap v_i\ne \emptyset$ and  $\phi(v_k)=\phi(v_j)$, or
\item $v_k\cap v_j\ne \emptyset$ and  $\phi(v_k)=\phi(v_i)$.
\end{enumerate}
As there can be at most one edge of each of these three types, we reach a contradiction.

Thus, no two of the $Q$-edges have a common vertex in $H$. Therefore, all $Q$-edges  must form a color class of $\phi$. Hence $G-Q=\CL(H')$, where $H'$ is obtained from $H$ by deleting all $Q$-edges, i.e., $G-Q$ is a proper $(k-1)$-color-line graph. 

For the second part, let $V(G)=Q_1\cup Q_2\cup\ldots \cup Q_k$ be a vertex clique partition of $G$ such that $G\setminus\bigcup_{1\le i\le k} E(Q_i)=L(H)$. Let $\phi$ be the proper edge coloring of $H$ as in the proof of Theorem~\ref{thm:proper-color-line2}: For each $i$, color all edges of $H$ corresponding to vertices in $Q_i$ with the same color $i$. Then $G=CL(H)$ with respect to $\phi$, and the $Q_i$-edges in $H$ form the color classes of $\phi$. Now, as we have seen in the first part, the $Q$-edges form a color class in $H$, hence $Q$ must be one of the cliques $Q_1, \ldots, Q_k$. 
\end{proof}

\smallskip
Note that $k+4$ in the previous Lemma is only sharp for $k=3$. For other values of $k$ slightly better bounds can be achieved, but we skip this step for the sake of better readability.

Now proper $k$-color-line graphs can be recognized as follows: Suppose that a given graph $G$ is a proper $k$-color-line graph. If $G$ has a maximal clique $Q$ with at least $k+4$ vertices, then, by Lemma~\ref{lem:key}, $G-Q$ is a proper $(k-1)$-color-line graph and we can reduce the problem to recognizing proper $(k-1)$-color-line graphs by using Lemma~\ref{lem:key} again. If $G$ is $K_{k+4}$-free, then, by Theorem~\ref{thm:proper-k-color-line}, $G$ can have at most $k\cdot (k+3)$ vertices and we may check all possible $O(k^{k\cdot(k+3)})$ (which is a constant provided $k$ is fixed) vertex partitions into at most $k$ cliques of $G$ if one satisfies the condition in Theorem~\ref{thm:proper-k-color-line}. Algorithm~\ref{alg:proper-k-color-line} gives a more detailed description.

\begin{algorithm}[!ht]
\DontPrintSemicolon
\KwIn{Graph $G=(V,E)$}
\KwOut{A graph $H$ with a proper edge $k$-coloring such that $G=\CL(H)$, or \lq $\mathtt{NO}$\rq\ if $G$ is not a proper $k$-color-line graph}

\medskip
$s := 0;\ R := G$\;
\While{$(s < k)$ and $R$ has a $K_{k-s+4}$}{
   compute a maximal clique $Q$ in $R$ with at least $k-s+4$ vertices\;
   $s:=s+1$\;
   $Q_s:=Q$\;
   $R:= R-Q_s$\;
}
\If{$|V(R)|> (k-s)(k-s+3)$}{
   \Return \lq $\mathtt{NO}$\rq
}
\If{$|V(R)|=0$}{
   $t:=s$
}
\eIf{$R$ has a vertex clique partition $V(R)=Q_{s+1}\cup\ldots\cup Q_t$, for some $s\le t\le k$, such that $G\setminus\bigcup_{1\le i\le t} E(Q_i)=L(H)$ for some graph $H$}{\label{lineno}
   color all edges of $H$ corresponding to vertices in $Q_i$ with the same color $i$, $1\le i\le t$\;
   \Return $H$
}{
   \Return \lq $\mathtt{NO}$\rq
}
\caption{Recognizing proper $k$-color-line graphs\label{alg:proper-k-color-line}}
\end{algorithm}  

As discussed above, the correctness of Algorithm~\ref{alg:proper-k-color-line} follows from Lemma~\ref{lem:key} and Theorem~\ref{thm:proper-k-color-line}. 

Let us roughly determine the running time of the algorithm. The while-loop will be terminated after at most $k$ rounds. Each round of the while-loop, the running time is dominated by line 3, which can be done in time $O(|V(R)|^{k-s+4}(|V(R)|+|E(R)|))$. Thus, the while-loop needs at most $k\cdot O(n^{k+4}(n+m))$ time in total. At line~\ref{lineno}, $|V(R)|\le (k-s)(k-s+3)$. Hence $R$ has at most $(t-s)^{(k-s)(k-s+3)}$ vertex partitions into $t-s\le k$, cliques. Thus, we have at most $O(k^{k(k+3)})$ partitions  $V(R)=Q_{s+1}\cup\ldots\cup Q_t$ for checking if $G\setminus\bigcup_{1\le i\le t} E(Q_i)=L(H)$, each of these tests can be done in $O(n+m)$ time by Theorem~\ref{thm:line2}. 
In summary, we conclude
\begin{theorem}
For fixed $k$, recognizing if an $n$-vertex $m$-edge graph $G$ is a proper $k$-color-line graph can be done in time $O(n^{k+4}(n+m))$. Moreover, if $G$ is a proper $k$-color-line graph, then a properly edge $k$-colored root graph $H$ of $G$ can be found in the same time complexity. 
\end{theorem}

\section{Proper $2$-color-line Graphs}\label{sec:proper-2-color-line}
We show in this section that proper $2$-color-line graphs admit a good characterization that allows us to recognize them in linear time. 
We first describe proper $2$-color-line graphs as follows. Recall that a co-bipartite graph is the complement of a bipartite graph.

\begin{proposition}\label{prop:proper-2-color-line}
The following statements are equivalent for any graph $G$.
\begin{itemize}
\item[\em (i)] $G$ is a proper $2$-color-line graph;
\item[\em (ii)] $G$ is a co-bipartite graph with a clique partition $V(G)=A\cup B$ such that each vertex in $A$ has at most two neighbors in $B$ and each vertex in $B$ has at most two neighbors in $A$; 
\item[\em (iii)] $G$ is a proper $2$-color-line graph of a bipartite graph (in which each component is a path or an even cycle).
\end{itemize} 
\end{proposition}
\begin{proof} 
(i) $\Rightarrow$ (ii): This implication follows immediately from the proof of Theorem~\ref{thm:proper-color-line2}: Let $G=\CL(H)$ with respect to a proper $2$-edge coloring $\phi:E(H)\to\{1,2\}$, and let $A=\{e\in E(H)\mid \phi(e)=1\}, B=\{e\in E(H)\mid \phi(e)=2\}$. Then $V(G)=A\cup B$ is a clique partition of $G$ and the bipartite graph $G\setminus \left(E(A)\cup E(B)\right)$ is a line graph, hence (ii) by noticing that a bipartite graph is a line graph if and only if its maximum degree is at most two (cf. Theorem~\ref{thm:line1}). 
 
\medskip\noindent
(ii) $\Rightarrow$ (iii):  Let $G$ admit a clique partition $V(G)=A\cup B$ such that each vertex in $A$ has at most two neighbors in $B$ and each vertex in $B$ has at most two neighbors in $A$. Then each component of the bipartite graph $G'=G\setminus \left(E(A)\cup E(B)\right)$ is a path or an even cycle. Hence $G'=L(H')$ for some bipartite graph $H'$ in which each component is a path or an even cycle, and the edge set of $H'$ is partitioned into two matchings corresponding to vertices in $A$ and in $B$, respectively. Color the edges in one matching with color $1$ and the edges in the other matching with color $2$, and let $H$ be the obtained properly edge $2$-colored bipartite graph. Then, clearly, $G=\CL(H)$. 

\medskip\noindent
(iii) $\Rightarrow$ (i): This implication is obvious.
\end{proof}

\smallskip
Since a co-bipartite graph may have many clique bipartitions, Proposition~\ref{prop:proper-2-color-line}~(ii) does not immediately yield an efficient recognition for proper $2$-color-line graphs. Theorem~\ref{thm:proper-2-color-line} below gives a good characterization; see Figure~\ref{fig:forbidden-proper-2-color-line} for the graphs $F_1, F_2, F_3$ and $F_4$. 

\begin{observation}\label{obs:no-proper-2-color-line}
None of $K_5-e, F_1, F_2, F_3$ and $F_4$ is a proper $2$-color-line graph. 
\end{observation}

\begin{proof} 
By case analysis, using Proposition~\ref{prop:proper-2-color-line} (ii). We note that every proper induced subgraph of $G\in\{K_5-e, F_1, F_2, F_3, F_4\}$ is a proper $2$-color-line graph. 
\end{proof}

\smallskip
We note that $K_5-e, F_1$, and $F_2$ are proper $3$-color-line graphs. 

\begin{theorem}\label{thm:proper-2-color-line}
$G$ is a proper $2$-color-line graph if and only if $G$ is a $(K_5-e, F_1, F_2, F_3, F_4)$-free co-bipartite graph.
\end{theorem}
\begin{figure}[H]
\begin{center}
\begin{tikzpicture}[scale=.5] 
\tikzstyle{vertex}=[draw,circle,inner sep=1.3pt] 
\node[vertex] (1) at (0,3) {};
\node[vertex] (2) at (2,2) {};
\node[vertex] (3) at (0,0) {};
\node[vertex] (4) at (5,3) {};
\node[vertex] (5) at (3,2) {};
\node[vertex] (6) at (5,0) {};
\node[vertex] (7) at (2.5,0.8) {};

\draw (1) -- (2) -- (3) -- (1); 
\draw (4) -- (5) -- (6) -- (4); 
\draw (7) -- (1); \draw (7) -- (2); \draw (7) -- (3); \draw (7) -- (4); \draw (7) -- (5); \draw (7) -- (6);

\draw(2.5,-1) node {$F_1$};
\end{tikzpicture}
\quad
\begin{tikzpicture}[scale=.5] 
\tikzstyle{vertex}=[draw,circle,inner sep=1.3pt] 
\node[vertex] (1) at (0,3) {};
\node[vertex] (2) at (2,2) {};
\node[vertex] (3) at (0,0) {};
\node[vertex] (4) at (5,3) {};
\node[vertex] (5) at (3,2) {};
\node[vertex] (6) at (5,0) {};
\node[vertex] (7) at (2.5,0.8) {};

\draw (1) -- (2) -- (3) -- (1); 
\draw (4) -- (5) -- (6) -- (4); 
\draw (3) -- (6); 
\draw (7) -- (1); \draw (7) -- (2); \draw (7) -- (3); \draw (7) -- (4); \draw (7) -- (5); \draw (7) -- (6);

\draw(2.5,-1) node {$F_2$};
\end{tikzpicture}
\quad
\begin{tikzpicture}[scale=.5] 
\tikzstyle{vertex}=[draw,circle,inner sep=1.3pt] 
\node[vertex] (1) at (0,3) {};
\node[vertex] (2) at (2,2) {};
\node[vertex] (3) at (0,0) {};
\node[vertex] (4) at (5,3) {};
\node[vertex] (5) at (3,2) {};
\node[vertex] (6) at (5,0) {};
\node[vertex] (7) at (2.5,0.8) {};

\draw (1) -- (2) -- (3) -- (1); 
\draw (4) -- (5) -- (6) -- (4); 
\draw (3) -- (6); \draw (2) -- (5);
\draw (7) -- (1); \draw (7) -- (2); \draw (7) -- (3); \draw (7) -- (4); \draw (7) -- (5); \draw (7) -- (6);

\draw(2.5,-1) node {$F_3$};
\end{tikzpicture}
\quad
\begin{tikzpicture}[scale=.5]
\tikzstyle{vertex}=[draw,circle,inner sep=1.3pt] 
\node[vertex] (1) at (0,3) {};
\node[vertex] (2) at (2,2) {};
\node[vertex] (3) at (0,0) {};
\node[vertex] (4) at (5,3) {};
\node[vertex] (5) at (3,2) {};
\node[vertex] (6) at (5,0) {};
\node[vertex] (7) at (2.5,0.8) {}; 

\draw (1) -- (2) -- (3) -- (1); 
\draw (4) -- (5) -- (6) -- (4); 
\draw (3) -- (6); \draw (2) -- (5); \draw (1) -- (4);
\draw (7) -- (1); \draw (7) -- (2); \draw (7) -- (3); \draw (7) -- (4); \draw (7) -- (5); \draw (7) -- (6); 

\draw(2.5,-1) node {$F_4$};
\end{tikzpicture}
\end{center}
\caption{Forbidden induced subgraphs for proper $2$-color-line graphs.}\label{fig:forbidden-proper-2-color-line}
\end{figure}
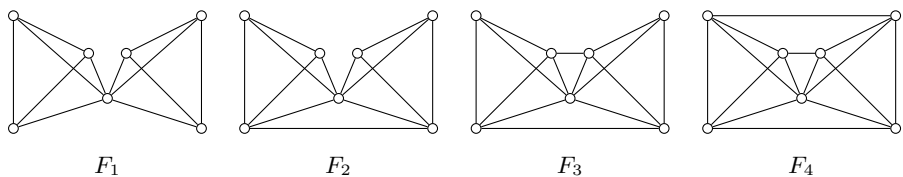
\begin{proof} 
\emph{Necessity}:  By Observation~\ref{obs:no-proper-2-color-line}, $G$ is $(K_5-e, F_1, F_2, F_3, F_4)$-free and by Proposition~\ref{prop:proper-2-color-line} (ii) $G$ is co-bipartite.

\medskip
\emph{Sufficiency}: Let $G$ be a $(K_5-e, F_1, F_2, F_3, F_4)$-free co-bipartite graph, and let $U$ be the set of all universal vertices of $G$. Let $A\cup B$ be a clique partition of the co-bipartite graph $G-U$. If $A=\emptyset$ or $B=\emptyset$, then $G$ is a complete graph, hence $G$ is a proper $2$-color-line graph ($G=\CL(M)$, where $M$ is a monochromatic matching). 

So, we may assume that $A$ and $B$ are not empty. 
Sine $U$ is the set of all universal vertices of $G$, we have 
\begin{equation*}
\text{for all $v\in A$, $B\setminus N(v)\not=\emptyset$ and, for all $v\in B$, $A\setminus N(v)\not=\emptyset$.}
\end{equation*}
Hence 
\begin{equation}\label{eq1}
\text{$|U|\le 2$.}
\end{equation} 
Otherwise, three vertices in $U$, a vertex $a\in A$ and a vertex in $B\setminus N(a)$ would induce a $K_5-e$. 
Moreover,
\begin{equation}\label{eq2}
\text{for all $v\in A$, $|N(v)\cap B|\le 2 - |U|$.}
\end{equation}
For, otherwise, $v$ and three vertices in $U\cup (N(v)\cap B)$ and a vertex in $B\setminus N(v)$ together would induce a $K_5-e$. By symmetry, 
\begin{equation}\label{eq3}
\text{for all $v\in B$, $|N(v)\cap A|\le 2 - |U|$.}
\end{equation} 
Thus, if $U=\emptyset$, then the clique bipartition $V(G)= A\cup B$ of $G$ satisfies Proposition~\ref{prop:proper-2-color-line} (ii), hence $G$ is a proper $2$-color-line graph, and we are done.

So, let us assume that $U\not=\emptyset$. Then
\begin{equation}\label{eq4}
\text{$|A|\le 2$ or $|B|\le 2$.}
\end{equation} 
Otherwise, by~(\ref{eq2}) and (\ref{eq3}) three vertices in $A$, three vertices in $B$ and a vertex in $U$ would induce an $F_i$ for some $i\in\{1,2,3,4\}$.

By (\ref{eq4}), let $|A|\le 2$, say. Then, by~(\ref{eq1}), (\ref{eq2}) and~(\ref{eq3}), the clique bipartition $V(G)=A\cup C$ of $G$ into cliques $A$ and $C=B\cup U$ satisfies the condition (ii) in Proposition~\ref{prop:proper-2-color-line}, hence $G$ is a proper $2$-color-line graph.  
\end{proof}

\smallskip
As co-bipartite graphs can be tested in linear time, Theorem~\ref{thm:proper-2-color-line} immediately yield an $O(n^7)$ time algorithm for recognizing if an $n$-vertex graph is a proper $2$-color-line graph, simply by checking all induced subgraphs on five and seven vertices. 
However, we can recognize proper $2$-color-line graphs in linear time by following the steps of the proof of Theorem~\ref{thm:proper-2-color-line}: 
\begin{enumerate}
 \item\label{step1} 
    Check if $G$ is co-bipartite; if not: stop with \lq \texttt{no}\rq, meaning $G$ is not a proper $2$-color-line graph.
 \item\label{step2} 
    Compute the set $U$ of universal vertices of $G$, and let $(A,B)$ be any clique bipartition of the co-bipartite graph $G-U$. 
 \item\label{step3}
    Let $G_1$ be the bipartite graph obtained from $G$ by deleting all edges in $A\cup U$ and in $B$, let $G_2$ be the bipartite graph obtained from $G$ by deleting all edges in $A$ and in $B\cup U$. (Note that in case $U=\emptyset$, $G_1=G_2$.)
 \item\label{step4}
    If neither $G_1$ nor $G_2$ is a line graph: stop with \lq \texttt{no}\rq. Otherwise stop with \lq \texttt{yes}\rq, meaning $G$ is a proper $2$-color-line graph. (If $G_i=L(H')$ for some $i=1,2$, then a properly edge $2$-colored graph $H$ with $G=\CL(H)$ can be obtained from $H'$ according to the proof of Proposition~\ref{prop:proper-2-color-line}, part (ii) $\Rightarrow$ (iii).)   
 
\end{enumerate}
Since each of steps~\ref{step1} -- \ref{step4} can be performed in linear time, the whole  procedure above has linear time complexity. The correctness follows directly from the proof of Theorem~\ref{thm:proper-2-color-line}. 
Thus, we conclude
\begin{theorem}\label{thm:reg-proper-2-color-line}
Given a graph $G$, it can be decided in linear time if $G$ is a proper $2$-color-line graph, and if so, a properly edge $2$-colored root graph $H$ for $G$ can be constructed in linear time.
\end{theorem}

\section{Conclusion}\label{sec:open}

Color-line and proper color-line graphs generalize the classical notion of line graphs. We have given Krausz-type characterizations for color-line and proper color-line graphs. These do not lead to efficient recognition, and the main problem concerning (proper) color line graphs is: 

\begin{itemize}
  \item[(1)] Determine the computational complexity of recognizing (proper) color-line graphs.
\end{itemize} 

\noindent 
A first step in solving (1) is to specify the number $k$ of colors in the (proper) color-line root graph. We have shown that, for any fixed $k\ge 2$, recognizing $k$-color-line graphs is NP-complete, while recognizing proper $k$-color-line graphs is polynomially. 
 
\begin{itemize}
  \item[(2)] Determine the computational complexity of recognizing proper $k$-color-line graphs, when $k$ is part of the input. 
\end{itemize}   

\noindent 
We remark that Algorithm~\ref{alg:proper-k-color-line} is an XP-algorithm solving (2). It is even unknown if (2) can be solved in time $f(k)\cdot n^{O(1)}$ for some computable function $f$ depending only on $k$, i.e., if (2) admits an FPT-algorithm with respect to parameter $k$.

We have found a forbidden induced graph characterization for proper $2$-color-line graphs, which allows us to recognize proper $2$-color-line graphs in linear time.

\begin{itemize}

  \item[(3)] Find a forbidden induced subgraph characterization for proper $k$-color-line graphs, $k\ge 3$.
\end{itemize}   

\noindent
Note that any line graph with $n$ vertices is a proper $n$-color-line graph. This leads to the following problem. 
\begin{itemize}
 \item[(4)] Given a line graph $G$, determine the smallest integer $k$ such that $G$ is a proper $k$-color-line graph.
\end{itemize}

\noindent
A related question is: What is the smallest integer $k$ such that a bridgeless cubic graph on $n$ vertices is a proper $k$-color-line graph? Recall that $3\le k\le n/2$ (cf. Corollary~\ref{cor:cubic}).   

\medskip
\noindent
\textbf{Acknowledgements.}\, 
Research of the second author was supported in part by NSF grant DMS-1600483. 
Thanks are also due to the unknown referees for their comments that greatly improved the presentation. 

\bibliographystyle{amsplain}

\end{document}